\documentclass[]{amsart}
\usepackage[utf8]{inputenc}
\usepackage{amsmath}
\usepackage{amsfonts}
\usepackage{amsthm}
\usepackage{amssymb}
\usepackage{color}
\usepackage{todonotes}
\usepackage{comment}
\usepackage{tikz}
\usepackage{hvfloat}
\usepackage{enumerate}
\usepackage{todonotes}
\usetikzlibrary{decorations.markings}
\usetikzlibrary{arrows}
\usepackage{graphicx}
\usepackage{caption}
\usepackage{subcaption}
\usepackage{mathtools}
\usepackage[all]{xy}

\usepackage{xcolor}

\makeatletter
\renewcommand{\boxed}[1]{\text{\fboxsep=.2em\fbox{\m@th$\displaystyle#1$}}}
\makeatother

\newcommand{\Z}{\mathbb{Z}}

\newcommand{\N}{\mathbb{N}}

\renewcommand{\le}{\leqslant}
\renewcommand{\ge}{\geqslant}

\theoremstyle{plain}
\newtheorem{thm}{Theorem}[section]
\newtheorem{lem}[thm]{Lemma}

\newtheorem{theorem}{Theorem}

\makeatletter
\newtheorem*{rep@theorem}{\rep@title}
\newcommand{\newreptheorem}[2]{%
\newenvironment{rep#1}[1]{%
\def\rep@title{#2 \ref{##1}}%
\begin{rep@theorem}}%
{\end{rep@theorem}}}
\makeatother
\newreptheorem{theorem}{Theorem}

\newtheorem*{thm*}{Theorem}
\newtheorem*{lem*}{Lemma}
\newtheorem*{prop*}{Proposition}
\newtheorem*{cor*}{Corollary}
\newtheorem*{qu*}{Question}
\newtheorem*{dt*}{Definition and Theorem}
\newtheorem*{exmp*}{Example}
\newtheorem*{exmps*}{Examples}
\newtheorem*{dprop*}{Definition and Proposition}
\newtheorem*{conj*}{Conjecture}

\theoremstyle{definition}
\newtheorem{defn}[thm]{Definition}

\newtheorem*{defn*}{Definition}
\newtheorem*{not*}{Notation}

\theoremstyle{plain}
\newtheorem{rem}[thm]{Remark}
\newtheorem*{rem*}{Remark}

\DeclareMathOperator\FSym{FSym}
\DeclareMathOperator\FAlt{FAlt}
\DeclareMathOperator\Sym{Sym}
\newcommand{\id}{\mathrm{id}}
\DeclareMathOperator\supp{supp}

\begin{document}
\title{Invariable generation and the Houghton groups}

\author{Charles Garnet Cox}
\address{School of Mathematics, University of Bristol, Bristol BS8 1UG, UK}
\email{charles.cox@bristol.ac.uk}

\thanks{}


\date{\today}
\begin{abstract}
The Houghton groups $H_1, H_2, \ldots$ are a family of infinite groups. In 1975 Wiegold showed that $H_3$ was invariably generated (IG) but $H_1\le H_3$ was not. A natural question is then whether the groups $H_2, H_3, \ldots$ are all IG. Wiegold also ends by saying that, in the examples he had found of an IG group with a subgroup that is not IG, the subgroup was never of finite index. Another natural question is then whether there is a subgroup of finite index in $H_3$ that is not IG. In this note we prove, for each $n\in \{2, 3, \ldots\}$, that $H_n$ and all of its finite index subgroups are IG.

The independent work of Minasyan and Goffer-Lazarovich in June 2020 frames this note quite nicely: they showed that an IG group can have a finite index subgroup that is not IG.
\end{abstract}
\maketitle
\section{Introduction}
A group $G$ is invariably generated by a subset $S$ if replacing each element of $S$ with any of its conjugates still results in a generating set for $G$. A group that has an invariable generating set is called invariably generated (IG). If a finite choice for $S$ exists, then we say that $G$ is finitely invariably generated (FIG). It can be that an infinite set $S$ is required, even if $G$ is finitely generated, proved independently in \cite{ashot, exam2} (and in this case setting $S=G$ may be natural). If no invariable generating set exists, we say that the group is not invariably generated ($\neg$IG). In 2014 Kantor, Lubotzky, and Shalev studied this notion for infinite linear groups, and since this paper (\cite{intro}) many papers have investigated invariable generation for infinite groups. Dixon in \cite{Dixon} is generally the reference for introducing the concept for finite groups, all of which are IG, but the idea is so natural that equivalent notions have been considered before this e.g. \cite{intro2} in 1872 and \cite{Wiegold1, Wiegold2} for work on infinite groups. Examples of recent work include \cite{intro} on linear groups, \cite{topgen1, topgen} on topological groups, \cite{Gel} on convergence groups, \cite{thompson} on the Thompson groups F, T, V, and \cite{iteratedwreath, Cox5} on wreath products.

In \cite{intro, Wiegold1} it was shown that IG and FIG are stable under finite extensions, and both query whether the properties IG and FIG are stable under taking finite index subgroups. The answer is no, as shown independently in \cite{ashot, exam2}. At first glance this may seem disheartening. But such examples tell us that IG and FIG are not as tame as they may initially appear. We could therefore follow the direction that occurred as a reaction to the existence of a finitely presented group with an unsolvable word problem (and other such decision problems). In the context of FIG and IG groups, this yields 3 natural questions. In some sense all of these are subquestions of ``Can we describe the classes of FIG and IG groups?''.
\begin{enumerate}
\item[(Q1)] Which specific classes of groups have the property that FIG or IG are stable under taking finite index subgroups?
\item[(Q2)] Which specific classes of groups have the property that every finitely generated IG group in the class is FIG?
\item[(Q3)] Is it possible to make characterisations for classes of groups which do not appear as answers to (Q1) and (Q2)?
\end{enumerate}

This note tackles (Q1) for the Houghton groups. Houghton introduced his groups in \cite{houintro}, and used them to answer a question on the number of relative ends of a group. Since then, numerous authors have investigated their properties e.g.\ finiteness properties \cite{brown}, the conjugacy problem \cite{ConjHou}, centraliser structure \cite{simon}, the twisted conjugacy problem \cite{Cox1}, and the $R_\infty$-property \cite{sigmatheory,Cox2}. Their conjugacy growth functions are exponential, as noted in \cite{Sapir}. They are also isolated points in the space of groups \cite{isolated}. Of most interest to the author is the combinatorial nature of this family of groups, which leads to bespoke arguments generally being required to answer questions about them. More detailed introductions can be found in \cite{ConjHou, Cox1}.

\begin{not*} For a non-empty set $X$, let $\Sym(X)$ denote the group of all bijections on $X$. Given any $g\in \Sym(X)$, let $\supp(g):=\{x\in X\;:\; (x)g\ne x\}$. Then $\FSym(X):=\{g\in\Sym(X)\;:\;|\supp(g)|<\infty\}$ and $\FAlt(X)\le\FSym(X)$ consists of all even permutations (so that if $|X|=n<\infty$ then $\FAlt(X)\cong A_n$ and $\FSym(X)\cong S_n$).
\end{not*}
Of use in this note are the facts that $\FSym(X)$ is generated by the set of all 2-cycles and $\FAlt(X)$ is generated by the set of all 3-cycles. These groups also provide nice examples $\neg$IG groups: conjugating $(a_1\;\ldots\;a_m)\in\FSym(X)$ by $g\in \Sym(X)$ yields $(a_1g\;\ldots\;a_mg)$, meaning that we can fix a specific $x\in X$ and impose that each conjugacy class representative that we pick is in the stabiliser of $x$.
\begin{not*} Let $\N:=\{1, 2, 3, \ldots\}$ and for any $n\in \N$ let $X_n:=\{1,\ldots, n\}\times\N$. For each $i\in\{1,\ldots, n\}$ let $R_i:=\{(i,m)\;:\;m\in\N\}\subseteq X_n$, the $i$th ray of $X_n$.
\end{not*}

\begin{figure}[h!]
\centering
\begin{tikzpicture}[scale=0.5]

\foreach \x in {90, 225, 315}
{
\draw(\x:4.5cm)--(\x:0.5cm);
\filldraw(\x:4cm)circle(2pt)--(\x:1cm)circle(2pt);
\filldraw(\x:3.5cm)circle(2pt)--(\x:1.5cm)circle(2pt);
\filldraw(\x:3cm)circle(2pt)--(\x:2cm)circle(2pt);
\filldraw(\x:2.5cm)circle(2pt)--(\x:0.5cm)circle(2pt);
\node [below] at (0.5, 4.5) {$R_1$};
\node [below] at (3.5, -2) {$R_2$};
\node [below] at (-3.5, -2) {$R_3$};
}
\end{tikzpicture} \caption{Visualising the set $X_3=R_1\cup R_2\cup R_3$.}
\end{figure}
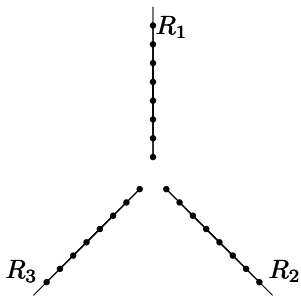
\begin{not*} Let $n\in\{2, 3, \ldots\}$ and $k\in\{2, \ldots, n\}$. Then $g_k$ is the element of $\Sym(X_n)$ with $\supp(g_k)=R_1\cup R_k$ and
\begin{equation*}
(i,m)g_k=\left\{\begin{array}{ll}(1, m+1) & \mathrm{if}\ i=1\ \mathrm{and}\ m\in\N\\ (1,1) & \mathrm{if}\ i=k\ \mathrm{and}\ m=1\\ (k, m-1) & \mathrm{if}\ i=k\ \mathrm{and}\ m\in\{2, 3, \ldots\}.
\end{array}\right.
\end{equation*}
\end{not*}
Originally the Houghton groups were the family $H_n$ where $n\in\{3, 4, \ldots\}$. In the following we also include the standard interpretations for $H_1$ and $H_2$.
\begin{defn*} Let $n\in\{3, 4,\ldots\}$. Then $H_n\le \Sym(X_n)$ with $H_n:=\langle g_2, \ldots, g_n\rangle$. We define $H_1:=\FSym(X_1)=\FSym(\N)$ and $H_2:=\langle g_2, ((1,1)\;(1,2))\rangle$. 
\end{defn*}
Note, given any countably infinite set $X$, that $\FSym(X)\cong \FSym(\N)$. It follows that $H_2=\FSym(X_2)\rtimes\langle g_2\rangle\cong \FSym(\Z)\rtimes \Z$. The following lemma is generally attributed to \cite{Wiegold2}, and the key observation is that the commutator of $g_2$ and $g_3$ is a 2-cycle.
\begin{lem*} Let $n\in\N$. Then $\FSym(X_n)\le H_n$.
\end{lem*}
Moreover, the elements of $H_n$ are exactly those of $\Sym(X_n)$ that are `eventually translations'. We can state this more formally, using \cite[Lem. 2.1]{ConjHou}.
\begin{not*} Let $n\in\{2, 3,\ldots\}$ and $g\in H_n$. For each $i\in \{1, \ldots, n\}$, let $t_i(g)\in\Z$ be chosen so that $(i,m)g=(i,m+t_i(g))$ for all but finitely many points $(i,m)\in R_i$. Let $\underline{t}(g):=(t_1(g),\ldots,t_n(g))\in\Z^n$ and for any $X\subseteq H_n$ let $\underline{t}(X):=\{\underline{t}(x)\;:\;x\in X\}$. Finally, for each $i\in \{1, \ldots, n\}$, let $z_i(g)$ be the smallest natural number such that $(i,m)g=(i,m+t_i(g))$ for all $m\in\{z_i(g), z_i(g)+1, \ldots\}$.
\end{not*}
In Section 2 we show the following, which is achieved by writing down specific generating sets and then proving that they are invariable generating sets. We use the standard notation, for a group $G$, that $d(G):=\inf\{|S|\;:\;\langle S \rangle=G\}$.
\begin{theorem} \label{mainthmA} Let $n\in \{2, 3, \ldots\}$. Then $H_n$, the nth Houghton group, is FIG. Furthermore, $H_n$ can be invariably generated by a set of size $d(H_n)$.
\end{theorem}

Recall that $G$ and $H$ are \emph{commensurable} if there exist $A\le G$ and $B\le H$ both of finite index where $A\cong B$. The main work of this note is to prove the following.
\begin{theorem} \label{mainthmB} Let $n\in \{2, 3, \ldots\}$ and $G$ be commensurable to $H_n$, the nth Houghton group. Then $G$ is FIG.
\end{theorem}
To prove this we note that, by \cite{intro, Wiegold1}, a finite extension of a FIG group is FIG. So in order to prove Theorem \ref{mainthmB}, it is sufficient to work with all finite index subgroups of $H_n$ for every $n\in\{2, 3,\ldots\}$. Such work may appear daunting. But the structure of the finite index subgroups of the Houghton groups is well understood. We take the following definition and lemma from \cite{Hou2}.
\begin{not*} Let $n\in\{2, 3, \ldots\}$ and $v\in 2\N$. Then $U_v:=\langle g_2^v, \ldots, g_n^v, \FAlt(X_n)\rangle$.
\end{not*}
We restrict ourselves to $v\in2\N$ since then $\FSym(X_n)\cap U_v=\FAlt(X_n)$.
\begin{lem*} Let $n\in\{2, 3, \ldots,\}$ and $U$ be a finite index subgroup of $H_n$. Then there exists a $d\in 2\N$ such that $U_d$ is finite index in $U$.
\end{lem*}
From this lemma we can prove Theorem \ref{mainthmB} by proving the following proposition, which is our sole aim in Section 3. A similar idea was used in \cite{Cox1}.
\begin{prop*} Let $n\in \{2, 3, \ldots\}$ and $v\in 2\N$. Then $U_v\le H_n$ is FIG.
\end{prop*}

\vspace{0.3cm}
\noindent\textbf{Acknowledgements.} I thank Kantor, Lubotzky, and Shalev whose paper, \cite{intro}, has brought about such diverse avenues of investigation regarding invariable generation of infinite groups.

\section{Invariable generation of the Houghton groups}
We work with two distinct cases: $n=2$ and $n>2$. The following remark will be useful throughout this note.
\begin{rem} \label{keyremark}Let $G$ be a group. For any $g\in G$, $\{s_i\;:\;i \in I\}$ generates $G$ if and only if $\{s_i^g\;:\;i \in I\}$ generates $G$. Therefore, when considering whether a generating set $S$ invariably generates $G$, without loss of generality we can assume that one $s\in S$ is conjugated to itself.
\end{rem}
\begin{not*} Let $t$ denote the transformation $t: z\mapsto z+1$ for all $z\in \Z$.
\end{not*}

\begin{defn} Given groups $G, H$ with $H\le G$ or $G\le H$, we say that $a, b \in G$ are $H$-\emph{conjugate} if there exists an $h\in H$ such that $h^{-1}ah=b$.
\end{defn}

\begin{lem} There exists a set, of size 3, that invariably generates $H_2$.
\end{lem}
\begin{proof} We work with $\FSym(\Z)\rtimes \langle t\rangle\cong H_2$. Let $S$ consist of
\begin{enumerate}[i)]
\item the transformation $t$
\item the element $t(0\;1)$, which consists an infinite orbit with support $\Z\setminus\{0\}$ and a single fixed point
\item the 2-cycle $(0\;1)$.
\end{enumerate}
By Remark \ref{keyremark}, we can assume that $t$ is fixed. Any $\Sym(\Z)$-conjugate of $t(0\;1)$ will fix some point $x\in\Z$ and have an infinite orbit with support equal to $\Z\setminus \{x\}$ (since conjugation by elements of $\Sym(X)$ preserves cycle-type). We can therefore conjugate this, using a suitable power of $t$, to an element $f$ which consists of an infinite orbit equal to $\Z\setminus\{0\}$ and which fixes 0. A conjugate of $(0\;1)$ will be equal to $(a\;b)$ for some $a, b \in \Z$. So conjugating $(a\;b)$ by $t^{b-a}$ will produce $(0\;b-a)$. Conjugating this 2-cycle by any power of $f$ will provide a 2-cycle of the form $(0\;c)$. Now, because $f$ acts transitively on $\Z\setminus\{0\}$, there will be a $d\in \Z$ such that $f^{-d}(0\;b-a)f^d=(0\;1)$. Hence $\langle t, f, (a\;b)\rangle=\langle (0\;1), t\rangle =\FSym(\Z)\rtimes \langle t\rangle$, as required.
\end{proof}
We now work to show that $H_2$ can be invariably generated by a set of size 2. Our first lemma is set up more generally so that it can be used for $H_n$ where $n\in\{2, 3,\ldots\}$.

\begin{lem}\label{3cycle} Let $n\in\{2, 3,\ldots\}$ and $\sigma\in\FSym(X_n)$. If $f\in H_n$ has $\underline{t}(f)=\underline{t}(g_2)$ and $\supp(\sigma)\subset \{(1, m)\;:\;m\ge z_1(f)\}$, then $\langle f, \sigma\rangle$ contains a 3-cycle.
\end{lem}
\begin{proof} Write $\sigma$ in disjoint cycle notation, so that $\sigma=\prod_{i=1}^k\sigma_i$. Also let $\supp(\sigma_1)$ contain $y$ where $y=(1,p)$ with $p:=\max\{m\in\N\;:\;(1, m)\in \supp(\sigma)\}$. Note that there is a $d\in\N$ such that $\supp(f^{-d}\sigma f^d)\cap\supp(\sigma)=y$. Hence conjugating $\sigma$ by $f^{-d}\sigma f^d$ we obtain $\alpha\in\FSym(R_1)$ where $(x)\alpha=(x)\sigma$ for all $x\in \supp(\sigma)\setminus \{y\}$. If $\sigma_1=(x_0\;x_1\;\ldots\;x_c\;y)$, then we can write $\alpha$ in disjoint cycle notation as $\alpha_1\prod_{i=2}^k\sigma_i$ where $\alpha_1=(x_0\;x_1\;\ldots\;x_c\;z)$ for some $z=(1,q)$ with $q>p$. Computing $\sigma\alpha^{-1}$ provides us with a 3-cycle.
\end{proof}

\begin{not*} Let $s=t(4\;3\;2\;1)$ and $s'$ denote a fixed choice of conjugate of $s$.
\end{not*}
Our aim is now to show that $\langle t, s'\rangle=\FSym(\Z)\rtimes\langle t\rangle$. Our approach will be to show that $\langle t, s'\rangle$ contains an odd permutation, that it contains $\FAlt(\Z)$, and then combine these to show that it contains $(0\;1)$.
\begin{lem}\label{oddperm} The group $\langle t, s'\rangle$ contains an element from $\FSym(\Z)\setminus\FAlt(\Z)$.
\end{lem}
\begin{proof} Clearly $s'=t^{-k}\alpha^{-1}s\alpha t^k$ using that any element of a semidirect product $\FSym(\Z)\rtimes\langle t\rangle$ decomposes as a product $\sigma t^l$ for some $\sigma\in\FSym(\Z)$ and $l\in \Z$. So $\alpha^{-1}s\alpha\in \langle t, s'\rangle$. Clearly $\alpha^{-1} s\alpha=\beta s$ for some $\beta\in\FAlt(\Z)$. Then, from our choice of $s$, we have that $st^{-1}\in\FSym(\Z)\setminus\FAlt(\Z)$ and so we also have that $\alpha^{-1}s\alpha t^{-1}=\beta st^{-1}\in \FSym(\Z)\setminus\FAlt(\Z)$.
\end{proof}
\begin{lem}\label{cycle012} The group $\langle t, s'\rangle$ contains $(0\;1\;2)$. In particular, $\FAlt(\Z)\le \langle t, s'\rangle$.
\end{lem}
\begin{proof} By Lemma \ref{oddperm}, $\langle t, s'\rangle$ contains an element of $\FSym(\Z)$, and so by conjugating this by a suitably large power of $t$ we see that $\langle t, s'\rangle$ contains an element $\sigma\in\FSym(\Z)$ so that Lemma \ref{3cycle} can be applied to $\sigma$ and $t$, meaning that $\langle t, s'\rangle$ contains a 3-cycle.

Let $a<b<c$ denote the fixed points of $s'$ (which has exactly 3 fixed points since cycle type is preserved by conjugacy in $\Sym(\Z)$ and $s$ has exactly 3 fixed points by contruction). If $b=c-1$, then set $a':=a$ and skip to the next paragraph. Otherwise, there exists a power, $k$, of $t$ such that $t^{-k}s't^k=:f_{c,d,e}$, an element with fixed points $c<d<e$ where $d-c=b-a$ and $e-d=c-b>1$. Next, conjugate $s'$ by a suitable power of $f_{c,d,e}$ to obtain $f_{a',c-1,c}$, an element with fixed points $a', c-1, c$. Such a power exists because $f_{c,d,e}$ acts transitively on the set $\Z\setminus\{c, d, e\}$ and $c-1\not\in\{c, d, e\}$.

We can assume that $a'\ne c+1$. If it did, we could replace $f_{a',c-1,c}$ with $tf_{a',c-1,c}t^{-1}$. Let $(x\;y\;z)$ denote the 3-cycle in $\langle t, s'\rangle$, where $x<y<z$. Now:
\begin{enumerate}[1.]
\item conjugate $(x\;y\;z)$ by a suitable power of $t$ to obtain $(c\;y'\;z')$ where $c<y'<z'$;
\item conjugate $(c\;y'\;z')$ by a suitable power of $f_{a',c-1,c}$ to obtain $(c\;c+1\;z'')$ where $z''\in\Z$;
\item conjugate $(c\;c+1\;z'')$ by a suitable power of $t^{-1}f_{a',c-1,c}t$ to obtain $(c\;c+1\;c+2)$. Note that $t^{-1}f_{a',c-1,c}t$ has fixed points $a'+1, c, c+1$ and acts transitively on $\Z\setminus\{a'+1, c, c+1\}$. Our assumption that $a'\ne c+1$ then means that $a'+1\ne c+2$.
\end{enumerate}
Finally, we can conjugate $(c\;c+1\;c+2)$ by a suitable power of $t$ to obtain $(0\;1\;2)$.
\end{proof}
\begin{lem} The group $H_2$ can be invariably generated by a set of size 2. In particular, $\langle t, s'\rangle=\FSym(\Z)\rtimes\langle t\rangle$.
\end{lem}
\begin{proof} We have that $\FAlt(\Z)\le \langle t, s'\rangle$ by Lemma \ref{cycle012}. Then, by Lemma \ref{oddperm}, $\langle t, s'\rangle$ also contains an element $\alpha\in\FSym(\Z)\setminus\FAlt(\Z)$. Hence there exists a $\beta\in\FAlt(\Z)$ such that $\alpha\beta$ is a 2-cycle, which we can conjugate to $(0\;1)$ using an element of $\FAlt(\Z)$. Finally, $\langle t, s'\rangle\ge \langle t, (0\;1)\rangle=\FSym(\Z)\rtimes\langle t\rangle$.
\end{proof}

\begin{rem*} A similar argument can be produced to show that $\langle t, u\rangle=\FSym(\Z)\rtimes \langle t\rangle$, where $u$ is a conjugate of $t(0\;1)$. This can be done by using $u$ and $t$ to produce an element with 2 fixed points.
\end{rem*}

For the rest of this section we will work with a fixed $n\in\{3, 4,\ldots\}$, and show that $H_n$ can be invariably generated by a set of size $n-1$. We start by introducing a set of size $n$ that will invariably generate $H_n$. This initial approach is similar to \cite{Wiegold2} in spirit, whereas the additional arguments for the smaller set are closer to the ideas used for $H_2$ above.

\begin{not*} Let $h:=g_ng_{n-1}\ldots g_2$, $\sigma:=(0\;1)$, and $S:=\{g_2, g_3,\ldots, g_{n-1}, h, \sigma\}$. For each $s\in S$, let $s'$ denote a conjugate of $s$ and let $S':=\{g_2', \ldots, g_{n-1}', h', \sigma'\}$. By using Remark \ref{keyremark}, we will assume that $h'=h$.
\end{not*}

\begin{defn} Two sets $\{x_2, \ldots, x_n\}$ and $\{y_2, \ldots, y_n\}$ are called \emph{translation equivalent} if $\underline{t}(x_i)=\underline{t}(y_i)$ for every $i \in \{2, \ldots, n\}$. Note for any $f_2,\ldots, f_n\in H_n$ that $\{f_2^{-1}x_2f_2,\ldots, f_n^{-1}x_nf_n\}$ is translation equivalent to $\{x_2, \ldots, x_n\}$ and that being translation equivalent is an equivalence relation.
\end{defn}

\begin{lem} \label{gen1} Let $X=\{x_2,\ldots, x_n\}$ be translation equivalent to $\{g_2, \ldots, g_n\}$. Then $X\cup \FSym(X_n)$ generates $H_n$.
\end{lem}
\begin{proof}
Let $G=\langle X\cup \FSym(X_n)\rangle$. Given an element $g \in H_n$, we note that $g\prod_{i=2}^nx_i^{t_i(g)}\in \FSym(X_n)$, and so lies in $G$. But then $x_2, \ldots, x_n \in G$, meaning that $g\in G$.
\end{proof}

\begin{not*} Let $d\in \N$. Then $R_1^{(d)}:=\{(1, m)\in R_1\;:\; m\ge d\}$.
\end{not*}

\begin{lem} \label{gen2} For every $d \in \N$ we have that $\FSym(X_n)\le\langle \FSym(R_1^{(d)}), h\rangle$.
\end{lem}
\begin{proof}
Let $\alpha \in \FSym(X_n)$. Then there exists a suitably large $k\in \N$ such that $h^{-k}\alpha h^k\in \FSym(R_1^{(d)})$, for example $k=d+\max\{z_i(\alpha)\;:\;i=1, \ldots, n\}$.
\end{proof}

The following allows us to replace $S'$ with one that is simpler to work with.

\begin{lem} \label{introT} There exists $\{h_2, h_3, \ldots, h_n, h, (a\;b)\} \subseteq \langle S'\rangle$, where:
\begin{enumerate}[i)]
\item $\{h_2, \ldots, h_n\}$ is translation equivalent to $\{g_2, \ldots, g_n\}$ and, for each $i\in \{2, \ldots, n\}$, that $\supp(h_i)\subseteq R_1\cup R_i$ and $(i,m)h_i=(i,m)g_i$ for all $m\in \N$;
\item $h=g_ng_{n-1}\ldots g_2$, as above; and
\item $a, b \in \supp(h_2)\cap\supp(h_3)$.
\end{enumerate}
\end{lem}
\begin{proof} By definition, $h \in S'$. Let $g_n':=hg_2^{-1}\ldots g_{n-1}^{-1}$. For the 2-cycle in $S'$ we can, by choosing a suitably large $k$, conjugate it by $h^k$ to obtain a 2-cycle arbitrarily far along the first branch of $X_n$. Similarly, for each $j\in\{2,\ldots, n\}$, there exists $k_j\in \N$ such that conjugating $g_j'$ by $h^{k_j}$ produces a suitable choice for $h_j$.
\end{proof}

The idea to prove the following is to show that we can conjugate the 2-cycle in $S'$ to an element $\alpha$ such that $\langle h_2, \alpha\rangle \cong H_2$.
\begin{lem} There exists a $d\in \N$ such that $\FSym(R_1^{(d)})\le \langle S'\rangle$.
\end{lem}
\begin{proof} It is sufficient to show that $\FSym(R_1^{(d)})\le \langle T\rangle$ for some $d\in \N$, where $T$ is a set $\{h_2, \ldots, h_n, h, (a\;b)\}\subset\langle S'\rangle$ satisfying conditions (i), (ii), and (iii) of Lemma \ref{introT}. We first use elements of $\langle T\rangle$ to conjugate $(a\;b)$ to a 2-cycle with support equal to $\{(2, 1), (3, 1)\}$.
\begin{enumerate}[1.]
\item From our choice of $a, b$, we may conjugate $(a\;b)$ by a suitably large power of $h_3$ to obtain $(x\;y)$ where $x\in\supp(h_2)$ but $y\not\in\supp(h_2)$.
\item Conjugate $(x\;y)$ by a power of $h_2$ to move $x$ to $(2,1)$. Note that $(2,1)\not\in\supp(h_3)$.
\item Conjugate $((2,1)\;y)$ by a power of $h_3$ to move $y$ to $(3,1)$.
\end{enumerate}
Conjugate $((2, 1), (3, 1))$ by a suitable power of $h$ to obtain $((1, d)\;(1, d+1))$, where $d\ge z_1(h_2)$. Then $\langle ((1,d)\;(1,d+1)), h_2\rangle=\FSym(\supp(h_2))\rtimes h_2$  which contains, by construction, $\FSym(R_1^{(d)})$.
\end{proof}

We have therefore shown, for $n\in\{2, 3, \ldots\}$, that $H_n$ is FIG. Our final aim in this section is to show, for any $n\in\{3, 4,\ldots\}$, that $H_n$ can be invariably generated by $\{g_2,\ldots, g_{n-1}, h\}$. We start by producing a version of Lemma \ref{oddperm} for $H_n$ where $n\in\{3, 4,\ldots\}$.

\begin{not*} Given a group $G$ and $g, h\in G$, let $[g,h]:=ghg^{-1}h^{-1}$.
\end{not*}

\begin{lem}\label{oddperm2} Given any $k\in\N$, the group $\langle S'\setminus\{\sigma'\}\rangle$ contains an element from $\FSym(R_1^{(k)})\setminus\FAlt(R_1^{(k)})$.
\end{lem}
\begin{proof} We know that $t(g_j')=t(g_j)$ for every $j\in \{2,\ldots, n\}$. So $g_j'g_j^{-1}\in\FSym(X_n)$, meaning $g_j'=\gamma_jg_j$ for some $\gamma_j\in\FSym(X_n)$. By using standard commutator identities, we see that $[\gamma_2g_2, \gamma_3g_3]\in\FSym(X_n)\setminus\FAlt(X_n)$. Now, given any $k\in \N$, conjugate $[\gamma_2g_2, \gamma_3g_3]$ by a suitably large power of $h$ to produce an element in $\FSym(R_1^{(k)})\setminus\FAlt(R_1^{(k)})$.
\end{proof}

\begin{lem} The set $\{g_2,\ldots, g_{n-1}, h\}$ invariably generates $H_n$.
\end{lem}
\begin{proof} For convenience let $G:=\langle S'\setminus\{\sigma'\}\rangle=\langle g_2',\ldots,g_{n-1}', h\rangle$. We will show that $G=H_n$. Applying the argument of Lemma \ref{introT} we see that we can construct elements $h_2, \ldots, h_n$ in $G$ enjoying the properties stated in the lemma. We therefore only need show that $G$ contains a 2-cycle to complete the proof. Taking a similar approach to the $H_2$ case, we will show that $\FAlt(X_n)\le G$ and then apply Lemma \ref{oddperm2} to obtain a 2-cycle.

Let $z:=\max\{z_1(h_2), z_1(h_3)\}$ and choose $a, b\in\N$ so that $(2, 2)h_2^a=(3, 2)h_3^b=(1,z)$. Then $\alpha:=h_2^ah_3^{-b}$ sends $(2,1)$ to $(3,1)$ and $(2, 2)$ to $(3,2)$, as well as potentially moving other points in $X_n$. Let $f_2$ denote the result of conjugating $h_2$ by $\alpha$, and note that $(2,1), (2,2)\not\in\supp(f_2)$. Let $r$ and $s$ denote the largest natural numbers such that $(2,r)\not\in\supp(f_2)$ and $(2, s), (2, s+1)\not\in\supp(f_2)$.

Using Lemma \ref{oddperm2} and then Lemma \ref{3cycle} we see that $G$ contains a 3-cycle, which we will denote by $\beta$. By conjugating $\beta$ by a suitably large power $k$ of $h$, we have that $\supp(h^{-k}\beta h^k)\subset R_1^{(z)}$, and we may then conjugate this 3-cycle by a suitably power of $h_2$ so that it has support $(2, r), (2, r+d), (2, r+e)$ where $r, d, e\in\N$ and $e>d$. By potentially using the inverse of this element, we can state that $\gamma_1:=((2, r)\;(2,r+d)\;(2,r+e))\in G$. Our key observation is now that $G$ contains, for every $k\in\N$, the element $\gamma_k:=((2,r)\;(2,r+d)\;(2,r+ke))$. To show this, for each $j\in\N$, let $\delta_j:=h_2^{je}\gamma_1^{-1}h_2^{-je}=((2,r+je)\;(2,r+(j+1)e)\;(2, r+d+je))$. We can then produce the elements $\{\gamma_k\;:\;k\in\N\}$ recursively using that $\delta_j^{-1}\gamma_j\delta_j=\gamma_{j+1}$ for every $j\in\N$.

By construction $(2, r+1), (2,r+d)\in\supp(f_2)$. We apply the following process to $\gamma_c$ and use that $c\in\N$ can be chosen to be arbitrarily large so that $r+ce$, $p_c$, and $q_c$ all exceed $z_2(f_2)$.
\begin{enumerate}[1.]
\item Conjugate $\gamma_c$ by a suitable power of $f_2$ to obtain $((2, r)\;(2,r+1)\;(2, p_c))=:\gamma_c'$.
\item Conjugate $\gamma_c'$ by a suitable power of $h_2$ to obtain $((2,s)\;(2,s+1)\;(2,q_c))=:\gamma_c''$. Note that $(2, s), (2, s+1)\not\in\supp(f_2)$ whereas $(2, q_c)\in R_2^{(z_2(f_2))}\subset \supp(f_2)$ by construction.
\item Conjugate $\gamma_c''$ by a suitable power of $f_2$ to obtain $((2,s)\;(2,s+1)\;(2,s+2))$.
\end{enumerate}
Now $G\ge \langle h_2, ((2,s)\;(2,s+1)\;(2,s+2))\rangle\ge \FAlt(R_1^{(z)})$. By Lemma \ref{oddperm2}, $G$ also contains an element in $\FSym(R_1^{(z)})\setminus\FAlt(R_1^{(z)})$. Multiplying this by a well chosen element in $\FAlt(R_1^{(z)})$ yields a 2-cycle $(a\;b)$ where $a, b\in R_1^{(z)}$ which is then $\FAlt(R_1^{(z)})$-conjugate to $((1,z)\;(1,z+1))$. Hence $\FSym(R_1^{(z)})\le G$ and Lemma \ref{gen2} and Lemma \ref{gen1} together imply that $G=H_n$.
\end{proof}

This concludes our proof of Theorem \ref{mainthmA}.

\section{Invariable generation for finite index subgroups of $H_n$}

Our aim is now to prove that the groups $U_v\le H_n$ are FIG in order to prove Theorem \ref{mainthmB}. We start by fixing some notation.

\begin{not*} Fix an $n \in \{2, 3, \ldots\}$ and a $v \in 2\N$ with $v\ne2$ (which we impose since then all 3-cycles in $A_{2v}$ are conjugate in $A_{2v}$). Throughout this section we work with $U_v\le H_n$ where $U_v:=\langle g_2^v,\ldots,g_n^v, \FAlt(X_n)\rangle$.
\end{not*}

We now introduce the set $S$ that we will prove invariably generates $U_v$. Some of the elements included may at this stage appear mysterious, but the motivation behind our choice for $S$ will be made clear as the argument unfolds.

\begin{not*} Let $S:=\{g_2^v, \ldots, g_{n-1}^v, h, \sigma_1, g_2^{2v}\sigma_1, g_2^{2v}\sigma_2\}$ where $h:=g_n^vg_{n-1}^v\ldots g_2^v$, $\sigma_1:=((1, 1)\;(1,2)\;(1,3))$, and $\sigma_2:=((1,1)\;(1,2))((1,3)\;(1,4)\;\ldots\;(1,2v))$. For each $s\in S$, let $s'$ denote a conjugate of $s$, and let $S':=\{s'\;:\;s\in S\}$. Remark \ref{keyremark} allows us to assume that $h'=h$.
\end{not*}

\begin{not*} Let $\Omega:=\{(1,1), (1,2),\ldots, (1,2v)\}$. Note that $\langle \sigma_1, \sigma_2\rangle=\FAlt(\Omega)$, where $\sigma_1$ and $\sigma_2$ are the elements introduced in the preceding notation.
\end{not*}

We now introduce notation for the elements in $S'$.

\begin{not*} For each $i\in \{2, \ldots, n-1\}$, let $u_i \in U_{v}$ denote an element such that $(g_i^v)'=u_i^{-1}g_i^vu_i$. Note that $\underline{t}(g_i^v)=\underline{t}((g_i^v)')$ implies that $\underline{t}(\langle S'\rangle)=\underline{t}(U_v)$. So there exist $w_2, \ldots, w_{n-1}\in \langle S'\rangle$ such that $\underline{t}(w_i)=\underline{t}(u_i)$ for each $i$. Now, for each $i$, let $h_i:=w_i(u_i)^{-1}g_i^vu_i(w_i)^{-1}$ and let $h_n:=h\prod_{i=2}^{n-1}(h_i^{-v})$. Then $h_2, \ldots, h_n, h \in \langle S'\rangle$ by construction, and also each $h_i$ is an $\FAlt(X_n)$-conjugate of $g_i^v$.
\end{not*}

\begin{rem} We have, for every $d\in\N$, that $\langle \FAlt(R_1^{(d)}), h_2, \ldots, h_n, h\rangle=U_v$ by running the same arguments used in Lemma \ref{gen1} and Lemma \ref{gen2}.
\end{rem}

\begin{not*} Let $\Omega_{S'}:=\{(1, p+1),(1, p+2), \ldots, (1, p+2v)\}$ where $p$ is the smallest non-negative integer satisfying both $p\equiv 0 \bmod(2v)$ and $p\ge z_1(h_2)$.
\end{not*}

\begin{lem} \label{reducetoomega} Let $X:= \FAlt(\Omega_{S'})\cup\{h_2, \ldots, h_n, h\}$. Then $\langle X\rangle =U_v$.
\end{lem}
\begin{proof} From the preceding remark, we need only show that $\FAlt(R_1^{(d)})\le \langle X\rangle$ for some $d\in\N$. Note, for every $m\in\{3, 4, \ldots\}$, that $((1,p+1)\;(1,p+2)\;(1,p+m))\in\langle\FAlt(\Omega_{S'}), h_2\rangle\le\langle X\rangle$. Now, from the simple calculations
\begin{equation*}
(1\;2\;m_1)(1\;2\;m_2)^{-1}=(2\;m_1\;m_2)\text{ and }(1\;2\;m_3)^{-1}(2\;m_1\;m_2)(1\;2\;m_3)=(m_1\;m_2\;m_3)
\end{equation*}
we have that every 3-cycle in $\FAlt(R_1^{(p+1)})$ lies in $\langle X\rangle$.
\end{proof}

Our issue now is that elements of $\FAlt(X_n)$ are too malleable under conjugacy in $U_m$. This is the reason for the elements $g_2^{2v}\sigma_1$ and $g_2^{2v}\sigma_2$ appearing in $S$, and we now motivate their introduction. We first recall some notation introduced in \cite[Defn. 2.4, 2.5]{ConjHou}, which provides language to adequately describe the infinite orbits of elements of $H_n$. We remove some of their generality, since it is not required for our needs.

\begin{defn*} We say two sets $A$ and $B$ are \emph{almost equal} if $|(A\cup B)\setminus (A\cap B)|<\infty$.
\end{defn*}

\begin{not*} Let $i\in \{1, 2\}$ and $r\in \{1, \ldots, 2v\}$. Then $R_{i,r}:=\{(i, 2vk+r)\;:\;k\in\N\}$.
\end{not*}

For any $\sigma\in\FAlt(\Omega_{S'})$, each infinite orbit of $g_2^{2v}\sigma$ is almost equal to
\begin{equation*}\label{eqnorbits} R_{2, 2v+1-r}\cup R_{1, (r)\sigma}
\end{equation*}
for some $r\in\{1,\ldots,2v\}$. Note that this also applies to $\sigma=\id$.
\begin{lem} \label{FSympreserves} Let $g, h \in H_n$ be $\FSym(X_n)$-conjugate. Then there exists a $d\in\N$ such that for each $x\in\bigcup_{i=1}^nR_i^{(d)}$ we can find an $e\in\N$ with $(x)g^{e+k}=(x)h^{e+k}$ for every $k\in\N$.
\end{lem}
\begin{proof} Write $g$ in disjoint cycle notation, so that $g\alpha^{-1}=\prod_{j=1}^a(\ldots x_{-1}^{(j)}\;x_0^{(j)}\;x_1^{(j)}\ldots)$ for some $a\in\N$ and $\alpha\in\FSym(X_n)$. Now $h=\omega^{-1}g\omega$ for some $\omega\in\FSym(X_n)$ and $\omega^{-1}g\omega=\omega^{-1}\alpha \omega \prod_{j=1}^a(\ldots (x_{-1}^{(j)})\omega\;(x_0^{(j)})\omega\;(x_1^{(j)})\omega\ldots)$.
\end{proof}

\begin{rem} Note that $\max_{i\in\{1, \ldots, n\}}\{|t_i(g)|, |t_i(h)|\}+\max_{i\in\{1, \ldots, n\}}\{z_i(g), z_i(h)\}$ is a suitable value for $d$ in the preceding lemma.
\end{rem}

The preceding lemma applies to $h_2$ and $g_2^{v}$. Thus it also applies to $h_2^2$ and $g_2^{2v}$. We now show that there are elements in $\langle S'\rangle$ that are $\FSym(X_n)$-conjugate to $g_2^{2v}\sigma_1$ and $g_2^{2v}\sigma_2$, and fix notation for these two elements.

\begin{lem} For $j\in\{1, 2\}$, $\langle S'\rangle$ contains an $\FAlt(X_n)$-conjugate of $g_2^{2v}\sigma_j$.
\end{lem}
\begin{proof} We note that $\{h_2,\ldots, h_n\}$ is translation equivalent to $\{g_2^v,\ldots, g_n^v\}$, and so for each $g\in U_v$ we can construct an element $w_g\in \langle h_2,\ldots, h_n\rangle$ such that $\underline{t}(w_g)= \underline{t}(g)$. Let $j\in\{1,2\}$. Then both $(g_2^{2v}\sigma_j)'=u_j^{-1}g_2^{2v}\sigma_ju_j$ and $w_{u_j}$ lie in $\langle S'\rangle$, meaning that $\langle S'\rangle$ contains an $\FAlt(X_n)$-conjugate of $g_2^{2v}\sigma_j$.
\end{proof}

\begin{not*} For $j\in\{1, 2\}$, let $s_{\sigma_j}$ denote a fixed choice of element in $\langle S'\rangle$ that is $\FAlt(X_n)$-conjugate to $g_2^{2v}\sigma_j$.
\end{not*}
The following lemma shows the usefulness of the elements $s_{\sigma_1}$ and $s_{\sigma_2}$.
\begin{lem} \label{keyrem} Let $d\in\N\cup\{0\}$ and $\sigma\in\FAlt(\Omega)$. Then there exists a $w\in\langle S'\rangle$ where, for each $i\in\{0, \ldots, d\}$, we have that $(x)w=(x)g_2^{-p-2vi}\sigma g_2^{p+2vi}$ for every $x\in\Omega_{S'}g_2^{2vi}$.
\end{lem}

\begin{proof} Let $z:=2v+\max\{z_1(s_{\sigma_1}), z_2(s_{\sigma_1}), z_1(s_{\sigma_2}), z_2(s_{\sigma_2}), z_1(h_2^{-1}), z_2(h_2^{-1})\}$ and $j\in\{1,2\}$. We will show that there exists $a, b \in \N$ such that for any $r\in\{1, \ldots, 2v\}$:
\begin{enumerate}[1.]
\item $(1, p+r)h_2^{-2a}=(2, 2vk+1-r)$ for some $k\in\N$ so that $(1, p+r)h_2^{-2a}\in R_2^{(z)}$;
\item $(1, p+r)h_2^{-2a}s_{\sigma_j}^b=(1, p+(r)\sigma_j+2cv)$ for some $c\in\N$; and so
\item $(1, p+(r)\sigma_j+2cv)h_2^{-2c}=(1, p+(r)\sigma_j)$.
\end{enumerate}
We start with $b\in\N$. Apply Lemma \ref{FSympreserves} to $s_{\sigma_j}$ and $g_2^{2v}\sigma_j$. Then, for all but finitely many $x\in R_{2, 2v+1-r}$, there exists a suitably large power of $s_{\sigma_j}$ sending $x$ to $R_{1, (r)\sigma_j}$. Similarly we can apply Lemma \ref{FSympreserves} to $h_2^2$ and $g_2^{2v}$ to find a suitable $a\in\N$. For $j\in\{1, 2\}$ and any $k_1, k_2, k_3\in\N$, let $w_j^{(k_1, k_2, k_3)}:=h_2^{-2k_1}s_{\sigma_j}^{k_2}h_2^{-2k_3}$.

We now note that we can replace $a$ and $b$ with $a'$ and $b'$ where $a':=a+k$ and $b':=b+2k$ for any $k \in \N$. Then $w_j^{(a',b',c)}\in\langle S'\rangle$ and, for every $i\in \{0, 1, \ldots, k\}$, we have that $(x) w_j^{(a',b',c)}=(x)g_2^{-p-2vi}\sigma_jg_2^{p+2vi}$ for every $x\in\Omega_{S'}g_2^{2vi}$. Now, because $\langle \sigma_1, \sigma_2\rangle=\FAlt(\Omega)$, we can produce the required element using a word made from powers of $w_1^{(k_1, k_2, k_3)}$ and $w_2^{(k_1, k_2, k_3)}$ where $k_1, k_2, k_3\in \N$ are chosen to be suitably large.
\end{proof}

\begin{lem} If $\langle S'\rangle$ contains $(x_1\;x_2\;x_3)$ with $x_1, x_2, x_3\in \Omega_{S'}$, then $\langle S'\rangle=U_v$.
\end{lem}
\begin{proof} Note that $(x_1\sigma\;x_2\sigma\;x_3\sigma)\in \langle S'\rangle$ for every $\sigma\in\FAlt(\Omega_{S'})$ by Lemma \ref{keyrem}. Hence $\langle S'\rangle$ contains all 3-cycles in $\FAlt(\Omega_{S'})$ and Lemma \ref{reducetoomega} yields the result.
\end{proof}

Our aim is now to show that there exists $(x_1\;x_2\;x_3)\in\langle S'\rangle$ with $x_1, x_2, x_3\in\Omega_{S'}$. We first show that this can be done if there exists $(a_1\;a_2\;a_3)\in\langle S'\rangle$ with $a_1, a_2, a_3$ on distinct orbits of $h_2^2$, and then show that such a 3-cycle does exist in $\langle S'\rangle$.
\begin{lem} If $\langle S'\rangle$ contains a 3-cycle $(a_1\;a_2\;a_3)$ with $a_1, a_2, a_3$ lying on distinct infinite orbits of $h_2^2$, then $\langle S'\rangle$ contains a 3-cycle with support contained in $\Omega_{S'}$.
\end{lem}

\begin{proof} After relabelling our points and possibly conjugating $(a_1\;a_2\;a_3)$ by a suitable power of $h_2^2$, we may assume that $a_1\in \Omega_{S'}$, $a_2 \in\Omega_{S'}h_2^{2k}$, and $a_3\in \Omega_{S'}h_2^{2(k+l)}$ where $k, l \in \N\cup\{0\}$. If $k=0$ then apply Lemma \ref{keyrem} to obtain an element $w \in \langle S'\rangle$ such that $(a_1)w=a_1'\in \Omega_{S'}\setminus\{a_1, a_2\}$, $(a_2)w=a_2$, and $(a_3)w=a_3$. Thus $w^{-1}(a_1\;a_2\;a_3)w=(a_1'\;a_2\;a_3)$ and we can conjugate $(a_1\;a_2\;a_3)$ by $(a_1'\;a_2\;a_3)$ to obtain $(a_1\;a_2\;a_1')$ where $a_1, a_1', a_2\in \Omega_{S'}$.

Our aim is to reduce to this case. Recall that $\Omega_{S'}=\{(1, p+1), \ldots, (1, p+2v)\}$. Using Lemma \ref{keyrem} we construct:
\begin{enumerate}[i)]
\item $w\in \langle S'\rangle$ with $(a_1) w = a$, $(a_2) w = b$, and $(a_3) w =c$ where $a=(1, p+1)$, $b=(1,p+v+1)h_2^{2k}$, and $c=(1,p+2v)h_2^{2(k+l)}$; and
\item $w_* \in \langle S'\rangle$ with $(a)w_*=(1,p+v+1)$, $(b)w_*=(1,p+1)h_2^{2k} $, and $w_*(c)=c$.
\end{enumerate}
Now $(a_1)ww_*h_2^{-1}w^{-1}=a_1$, whereas $(a_2) ww_*h_2^{-1}w^{-1}\in \Omega_{S'}h_2^{2(k-1)}$. Thus, by repeatedly applying these steps, we can reduce to the case where $k=0$.
\end{proof}

\begin{lem} There exists a 3-cycle $(a_1\;a_2\;a_3)$ in $\langle S'\rangle$ with $a_1, a_2, a_3$ lying on distinct $h_2^2$-orbits.
\end{lem}

\begin{proof} Choose a suitably large $k$ so that $\supp(h^{-k}\sigma_1'h^k)\subset R_1^{(p+1)}$. Then $h^{-k}\sigma_1'h^k=(a\;b\;c)\in\langle S'\rangle$ for some $a, b, c\in R_1^{(p+1)}$. We first consider the case where $b$ and $c$ lie on the same infinite orbit  $\mathcal{O}_z$ of $h_2^2$ and $a$ lies on a different infinite orbit $\mathcal{O}_y$ of $h_2^2$. Using Lemma \ref{keyrem} we can construct a $w\in \langle S'\rangle$ such that $(b)w=b$ and $(c)w=c$ but $(a)w=a'\in \mathcal{O}_{z'}$ where $\mathcal{O}_{z'}$ is an infinite orbit of $h_2^2$ distinct from $\mathcal{O}_y$ and $\mathcal{O}_z$. Hence $w^{-1}(a\;b\;c)w=(a'\;b\;c)$ and conjugating $(a\;b\;c)$ by $(a'\;b\;c)$ gives $(a\;b\;a')$, where $a, a'$, and $b$ lie on distinct orbits of $h_2^2$.

To deal with the second case where $a$, $b$, and $c$ all lie on the same infinite orbit $\mathcal{O}_y$ of $h_2^2$, we reduce this to our first case. Define $f:X_n\rightarrow \{0, 1\}$ by $f(x)=1$ if $x\in \supp(h_2)\setminus\mathcal{O}_y$ and $f(x)=0$ otherwise. Note that $f(xg)=f(x)$ for every $g\in\langle h_2^2\rangle$ and every $x\in X_n$. Let $z:=2v+\max\{z_1(h_2^{\pm 2}), z_2(h_2^{\pm 2}), z_1 (s_{\sigma_2}^{\pm 1}), z_2(s_{\sigma_2}^{\pm 1})\}$. First send $a, b$, and $c$ to $R_2^{(z)}$ by using a suitably large power of $h_2^{-2}$. Label these points as $(2, m_a)$, $(2, m_b)$, and $(2, m_c)$ where $m_a<m_b<m_c$. By construction $f((2, m_a))=f((2, m_b))=f((2, m_c))=0$. We can then consider the sequence
\begin{equation*}
\left(f\big((2,m_a)s_{\sigma_2}^k\big)\right)_{k\in \N\cup\{0\}}
\end{equation*}
where $s_{\sigma_2}^0:=\id$. Thus our sequence has first term zero. Importantly, this sequence only has finitely many terms equal to zero. This is because $\supp(\sigma_2)=\Omega$, and so there exists a suitably large power of $s_{\sigma_2}$ sending $(2, m_a)$ to an infinite orbit of $h_2^2$ distinct from $\mathcal{O}_y$. Let
\begin{equation*}
q:=\min\{k\in\N\;:\; f\big((2,m_x)s_{\sigma_2}^k\big)=1\}
\end{equation*}
and let $p_d:=(2, m_d)s_{\sigma_2}^q$ for each $d\in\{a, b, c\}$. We can then use a suitably large $k\in 2\N$ such that $p_ah_2^{-k}\in R_2^{(z)}$, $p_bh_2^{-k}\in R_2^{(z)}\cup\{p_b\}$, and $p_ch_2^{-k}\in R_2^{(z)}\cup\{p_c\}$. These sets must include $p_b$ and $p_c$ because there is no guarantee that $p_b, p_c\in \supp(h_2^2)$. Using a suitably large $l>q$ then ensures that $p_d':=(p_d)h_2^{-k}s_{\sigma_2}^{-l}\in R_2^{(z)}$ for each $d\in\{a, b, c\}$ where $p_a'\in \supp(h_2^2)\setminus \mathcal{O}_y$ but $p_b', p_c'\in \mathcal{O}_y$.
\end{proof}
Thus $\langle S'\rangle=U_v$ and we have proven Theorem \ref{mainthmB}.
\bibliographystyle{amsalpha}
\def\cprime{$'$}
\providecommand{\bysame}{\leavevmode\hbox to3em{\hrulefill}\thinspace}
\providecommand{\MR}{\relax\ifhmode\unskip\space\fi MR }
\providecommand{\MRhref}[2]{%
 \href{http://www.ams.org/mathscinet-getitem?mr=#1}{#2}
}
\providecommand{\href}[2]{#2}

\end{document}